\documentclass[12pt]{amsart}
\usepackage{amssymb, young, mathrsfs}
 \usepackage{tikz-cd}
\usetikzlibrary{matrix,arrows}

\renewcommand{\subsection}[1]{
\refstepcounter{subsection}\noindent{\bf \thesubsection.} }

\newcommand{\np}[1]{
\refstepcounter{subsection}\noindent{\bf \thesubsection.} }

\numberwithin{equation}{section}

\renewcommand{\geq}{\geqslant}
\renewcommand{\leq}{\leqslant}

\newcommand{\idlim}{\varprojlim} 
\newcommand{\E}{\mathrm{E}}                        
                            
\newcommand{\Span}{\operatorname{span}}                

\newcommand{\End}{\operatorname{End}} 

\newcommand{\CC}{\mathbb{C}} 
\newcommand{\ZZ}{\mathbb{Z}} 

\newtheorem{theorem}{Theorem}[section]

\newtheorem{proposition}[theorem]{Proposition}
\theoremstyle{definition}

\begin{document}

\title{Comments related to infinite wedge representations}

\author{Nathan Grieve}

\thanks{\emph{Mathematics Subject Classification (2010):} 05E05, 37K10.}
\thanks{This article has been published by Comptes Rendus Math\'{e}matiques de l'Acad\'{e}mie des Sciences. La
              Soci\'{e}t\'{e} Royale du Canada. Mathematical Reports of the Academy
              of Science. The Royal Society of Canada.  The final published version is [C. R. Math. Acad. Sci. Soc. R. Can. 39 (2017), no. 1, 13--35].}

\maketitle
\begin{abstract} 
We study the infinite wedge representation and show how it is related to the universal central extension of $g[t,t^{-1}]$, the loop algebra of a complex semi-simple Lie algebra $g$.  We also give an elementary proof of the boson-fermion correspondence.  Our approach to proving this result is based on a combinatorial construction combined with an application of the Murnaghan-Nakayama rule.
\end{abstract}

\section{Introduction}

In this article, we make two remarks about the \emph{infinite wedge representation}. To describe what we do let  $\mathbf{gl}(\infty)$ denote the Lie algebra of $\ZZ\times \ZZ$ \emph{band infinite matrices}.  Then $\mathbf{gl}(\infty)$ consists of those matrices $A =(a_{ij})_{i,j\in\ZZ}$ with $a_{ij}\in \CC$ and $a_{ij}=0$ for all $|i-j| \gg 0$.   Next
let $\widehat{\mathbf{gl}}(\infty)$ denote the Lie algebra determined by the $2$-cocycle $\mathrm{c}(\cdot,\cdot)$ of $\mathbf{gl}(\infty)$ with values in the trivial $\mathbf{gl}(\infty)$-module $\CC$:
$$\mathrm{c}(A,B):=  \sum_{i \leq 0, k>0} a_{ik}b_{ki} - \sum_{i>0,k\leq 0} a_{ik}b_{ki} \text{,} $$
see for instance \cite[p.~ 12]{Bloch:Okounkov:2000} or \cite[p.~ 115]{Kac:infinite:Lie:Algebras}.
The \emph{infinite wedge representation} is a suitably defined, see \S \ref{4} for precise details, Lie algebra representation 
$\rho : \widehat{\mathbf{gl}}(\infty) \rightarrow \End_\CC(F)\text{;}$
here $F$ is the \emph{infinite wedge space} that is the $\CC$-vector space determined by the set $\mathscr{S}$ which consists of those ordered strictly decreasing sequences of integers $S=(s_1,s_2,\dots)$, $s_i \in \ZZ$, with the properties that $s_i = s_{i-1}-1$ for all $i \gg 0$.

To describe our first theorem, let $g$ be a finite dimensional semi-simple complex Lie algebra,  $g[t,t^{-1}]$ its loop algebra and $\widehat{g}$ the universal (central) extension of $g[t,t^{-1}]$ in the sense of Garland \cite[\S 2]{Garland:1980} (see also \cite[\S 7.9]{Weibel}).  In Theorem \ref{universal:extension:theorem} we show how the representation  $\rho$ is related to $\widehat{g}$.

Our second theorem, Theorem \ref{bosonic:Extension:theorem}, gives an elementary proof of the \emph{boson-fermion correspondence}, in the sense of Kac-Raina-Rozhkovskaya \cite[Lecture 5, p. 46]{kac-raina-rozhkovskaya}.
To place this theorem in its proper context, let $\mathfrak{s}$ denote the \emph{oscillator algebra}, which is the universal extension of $\CC[t,t^{-1}]$ the loop algebra of the abelian Lie algebra $\CC$.  The Lie algebra $\mathfrak{s}$  is faithfully represented in $\End_\CC(F)$ and also in $\End_\CC(B)$, where $B$ denotes the polynomial ring in countably many variables with coefficients in the ring of Laurent polynomials.   The boson-fermion correspondence, as formulated in \cite[Lecture 5, p. 46]{kac-raina-rozhkovskaya} compare also with \cite[\S 14.9--14.10]{Kac:infinite:Lie:Algebras}, concerns extending these representations to all of $\widehat{\mathbf{gl}}(\infty)$ in such a way that an evident $\CC$-linear isomorphism $F\rightarrow B$ becomes an isomorphism of $\widehat{\mathbf{gl}}(\infty)$-modules; see \S \ref{6} for more precise details.  The traditional approach for proving this result is by way of vertex-operators, see \cite{Kac:infinite:Lie:Algebras} and \cite[Lecture 6, p.~ 46]{kac-raina-rozhkovskaya}.  The key point to our approach, which does not require the use of vertex-operators, is a combinatorial construction related to partitions, see \S \ref{fermi:gl}, together with the Murnaghan-Nakayama rule which we recall in  \S \ref{2.6}.

\noindent
{\bf Acknowledgements.}  This work benefitted from discussions with Jacques Hurtubise and John Harnad.  It was completed while I was a postdoctoral fellow at McGill University and also while I was a postdoctoral fellow at the University of New Brunswick where I was financially supported by an AARMS postdoctoral fellowship.

\section{Preliminaries}\label{2}

In this section, to fix notation and terminology for what follows, we recall a handful of combinatorial and Lie theoretic concepts.  For the most part we use  combinatorial terminology and conventions similar to that of \cite[I \S 1 --  5]{Macdonald:Sym:Func} and Lie theoretic terminology and conventions similar to that of \cite[\S 7 and \S 14]{Kac:infinite:Lie:Algebras}.

\np{}\label{2.1}   
Let $\mathscr{P}$ denote the set of partitions.  Then $\mathscr{P}$ consists of those infinite weakly decreasing sequences of non-negative integers of the form
$\lambda=(\lambda_1,\lambda_2,\dots)$ with the property that at most finitely many of the $\lambda_i$ are nonzero.  
If $\lambda=(\lambda_1,\lambda_2,\dots) \in \mathscr{P}$, then the number 
$\operatorname{weight}(\lambda):= \sum_{i=1}^{\infty}\lambda_i$ is called the \emph{weight of $\lambda$} and we denote by $\mathscr{P}_d$ the set of partitions of weight $d$.  If $\lambda=(\lambda_1,\lambda_2,\dots) \in \mathscr{P}$, then we often identify $\lambda$ with the finite weakly decreasing sequence $(\lambda_1,\dots, \lambda_r)$, where 
$r = \operatorname{length}(\lambda) := \max \{ i : \lambda_i \not = 0 \}.$

\np{}\label{2.2}
The \emph{Young diagram} of a partition $\lambda=(\lambda_1,\lambda_2,\dots) \in \mathscr{P}$ is defined to be the set of points $(i,j) \in \ZZ^2$ such that $1 \leq j \leq \lambda_i$.  When drawing the Young diagram associated to a partition we use the convention that the first coordinate is the row index, starts at $1$ and increases from left to right.  Similarly, the second coordinate is the column index, starts at $1$ and increases downward.  We refer to the  elements $(i,j)$ of a Young diagram as the \emph{boxes} of the associated partition and the entires $i$ and $j$ as the sides of the box.  

\np{}\label{2.3} If $\lambda, \mu \in \mathscr{P}$ and $\lambda \supseteq \mu$, so $\lambda_i \geq \mu_i$ for all $i \geq 1$, then the set theoretic difference of the Young diagrams corresponding to $\lambda$ and $\mu$ is denoted by  $\lambda \setminus \mu$ and is called a \emph{skew diagram}.
If $\lambda, \mu \in \mathscr{P}$ and $\lambda \supseteq \mu$, then let $\theta := \lambda \setminus \mu$ denote the skew diagram that they determine.  By a \emph{path} in  $\theta$, we mean a sequence $x_0,x_1,\dots,x_m$ with $x_i \in \theta$, such that $x_{i-1}$ and $x_i$ have a common side for $1\leq i\leq m$.  A subset $\nu \subseteq \theta$ is said to be \emph{connected} if every two boxes in $\nu$ can be connected by a path in $\nu$.   
The \emph{length} of $\theta$ is defined to be the number of boxes that appear in its diagram and is denoted by $\#\theta$.
We say that $\theta$ is a \emph{border strip} if it is connected and if it contains no $2 \times 2$ box.   Finally, if $\theta$ is a border strip, then we denote its \emph{height}  by $\operatorname{height}(\theta)$ and define it to be one less then the number of rows that it occupies.

\np{}\label{2.4}
The symmetric group $S_n$ acts on the polynomial ring $\CC[x_1,\dots,x_n]$ by permuting the variables and we let $\Lambda_n$ denote the subring of invariants.  We then have that 
$\Lambda_n = \bigoplus_{k\geq 0} \Lambda_n^k$ where $\Lambda_n^k\subseteq \Lambda_n$ is the subspace of symmetric polynomials of degree $k$.  If $k\in \ZZ_{\geq 0}$, $m,n\in \ZZ_{\geq 1}$ and $m\geq n$, we have evident restriction maps
$ \rho_{m,n}^k : \Lambda_m^k \rightarrow \Lambda_n^k\text{;}$ 
let 
$ \Lambda^k = \idlim \Lambda^k_n$
and 
$ \Lambda = \bigoplus_{k\geq 0} \Lambda^k\text{.}$
Then $\Lambda = \CC[h_1,h_2,\dots]$ where the $h_k$ are such that their image in $\Lambda^k_n$ is the $k$th complete symmetric function in the variables $x_1,\dots, x_n$.

\np{}\label{2.5}
Let $H(Z) := \sum_{k\geq 0} h_k Z^k \in \Lambda[[Z]]$ and define $p_k \in \Lambda$ by the coefficient of $Z^{k-1}$ in the power series
$P(Z) := H'(Z)/H(Z)\text{.} $
The image of each $p_k$ in $\Lambda^k_n$ is the $k$th power sum in the variables $x_1,\dots,x_n$ and the $h_k$ can be expressed in terms of the $p_k$ via the equality of power series
$H(Z) = \exp\left(\sum_{k\geq 1} t_k Z^k \right)\text{;} $
here $t_k = p_k/k$.  The Schur functions $s_\lambda$, defined for all partitions $\lambda = (\lambda_1,\lambda_2,\dots) \in \mathscr{P}$, are defined by 
$s_\lambda := \operatorname{det} \left( h_{\lambda_i - i + j}\right)_{1\leq i,j \leq n} \text{,}$
where $n = \operatorname{length}(\lambda)$, and form a $\CC$-basis for $\Lambda$.  In what follows we let $\langle \cdot , \cdot \rangle$ denote the symmetric bilinear form on $\Lambda$ for which the Schur polynomials are orthonormal.  In particular,
$ \langle s_\lambda, s_\mu \rangle = \delta_{\lambda,\mu} \text{.}$ 

\np{}\label{2.6}
By abuse of notation, we let $p_k \in \End_\CC(\Lambda)$ be the $\CC$-linear endomorphism given by multiplication by $p_k$.  The adjoint of $p_k$ with respect to $\langle \cdot, \cdot \rangle$, which we denote by $p_k^\perp$, is the $\CC$-linear endomorphism given by the differential operator $k \frac{\partial}{\partial p_k}$,  \cite[p.~ 76]{Macdonald:Sym:Func}.

The effect of the operator $p_k$ in the basis of Schur polynomials is given by the Murnaghan-Nakayama rule:
\begin{equation}\label{Murnaghan:Nakayama} p_k s_\lambda = \sum_{\substack{\nu \supseteq \lambda,  \\ \nu \setminus \lambda \text{ is a border strip} \\ \text{of length $k$}} }(-1)^{\operatorname{height}(\nu \setminus \lambda)} s_\nu \text{,} \end{equation}
\cite[p.~ 601]{Okounkov:Vershik:1996}.
Using \eqref{Murnaghan:Nakayama}, in conjunction with \cite[I.V. Ex. 3, p.~ 75]{Macdonald:Sym:Func}, we deduce the adjoint form of the Murnaghan-Nakayama rule:
\begin{equation}\label{Murnaghan:Nakayama:adjoint} 
p_k^{\perp} s_\lambda = \sum_{\substack{ \lambda \supseteq \nu, \\ \lambda \setminus \nu \text{ is a border strip} \\ \text{of length $k$}}} (-1)^{\operatorname{height}(\lambda \setminus \nu)} s_\nu\text{.} \end{equation}

\np{}\label{2.7}
We let $\mathrm{Mat}(\infty)$ denote the $\CC$-vector space of $\ZZ \times \ZZ$ matrices with entries in $\CC$.  If
$A=\left(a_{ij}\right)_{i,j\in\ZZ}, B=\left(b_{ij}\right)_{i,j\in\ZZ} \in \mathrm{Mat}(\infty)$
and $a_{ik} b_{kj} = 0$, for all $i,j\in\ZZ$ and almost all $k \in \ZZ$, then their product is given by 
$C = AB := \left( c_{ij}\right)_{i,j \in \ZZ}\text{,}
$
where
$c_{ij} = \sum_{k \in \ZZ} a_{ik}b_{kj}\text{.} $  Further, we
let $\E_{ij}$ denote the element of $\mathrm{Mat}(\infty)$ with $i,j$ entry equal to $1$ and all other entries equal to zero.  We say that a matrix $A=(a_{ij})_{i,j\in\ZZ} \in \operatorname{Mat}(\infty)$ is a \emph{band infinite matrix} if $a_{ij}=0$ for all $|i-j|\gg0$.  We denote the collection of band infinite matrices by ${\bf gl}(\infty)$ and regard it as a Lie algebra with Lie bracket given by
$[A,B]=AB-BA \text{.}$  We often express elements of ${\bf gl}(\infty)$ as infinite sums of matrices.  For example, the identity matrix
$1_{\ZZ \times \ZZ}=\left( \delta_{ij} \right)_{i,j\in \ZZ}$,  can be expressed as
$1_{\ZZ \times \ZZ} = \sum_{p\in \ZZ} \mathrm{E}_{pp}\text{.} $  Also every element of ${\bf gl}(\infty)$ can be written as a finite linear combination of matrices of the form $\sum_i a_i \E_{i,i+k}$, where $k \in \ZZ$ and $a_i \in \CC$.

\np{}\label{2.9}
Let  $\mathfrak{gl}_N[t,t^{-1}] :=\CC[t,t^{-1}]\otimes_\CC \mathfrak{gl}_N(\CC)$ which we regard as a Lie algebra with Lie bracket determined by 
$$[f(t)\otimes A, g(t) \otimes B ]=  f(t)g(t) \otimes [A,B] \text{.}$$  If $t^m \otimes e_{ij}$, for $i,j = 1,\dots, N$ and $m \in \ZZ$, denotes the standard basis elements of $\mathfrak{gl}_N[t,t^{-1}]$,  we then have that
$$[t^m \otimes e_{ij}, t^n \otimes e_{k \ell}]=
t^{m+n} \otimes \left(\delta_{jk}e_{i\ell} - \delta_{\ell i} e_{kj} \right), $$
and that the map
\begin{equation}\label{loop:incl}\iota_N : \mathfrak{gl}_N[t,t^{-1}] \rightarrow \mathbf{gl}(\infty)\text{,}
\end{equation} determined by 
$$t^m \otimes e_{ij} \mapsto \sum_{k \in \ZZ} \E_{N(k-m)+i,Nk+j}\text{,}$$ is a monomorphism of Lie algebras.   The image of $\iota_N$ is the Lie algebra of \emph{$N$-periodic band infinite matrices}, that is those $\ZZ \times \ZZ$ band infinite matrices $A = (a_{ij})_{i,j \in \ZZ}$ for which 
$a_{i+N,j+N}=a_{ij}\text{, for all $i,j \in \ZZ$.} $

\np{}\label{2.10} To define the Lie algebra $\widehat{\mathbf{gl}}(\infty)$, first let 
$$J := \sum_{m \leq 0} \mathrm{E}_{mm} - \sum_{m > 0} \mathrm{E}_{mm} \in {\bf gl}(\infty)$$
and observe that if $A$ and $B$ are elements of ${\bf gl}(\infty)$, then the matrix $[J,A]B$ has at most finitely many nonzero diagonal elements and the expression
$ \frac{1}{2}\operatorname{tr}\left([J,A]B \right)$
is a well defined element of $\CC$.  In particular, we have
\begin{equation}\label{eqn2.4}
\frac{1}{2}\operatorname{tr}\left([J,A]B \right) = \sum_{i\leq 0, k>0} a_{ik}b_{ki} - \sum_{i>0, k\leq 0} a_{ik}b_{ki},
\end{equation}
and we define the Lie algebra $\widehat{\mathbf{gl}}(\infty)$ to be the central extension determined by the following $2$-cocycle of $\mathbf{gl}(\infty)$ with values in the trivial $\mathbf{gl}(\infty)$-module $\CC$:
\begin{equation}\label{eqn2.5}\mathrm{c}(A,B):= \frac{1}{2}\operatorname{tr}\left([J,A]B \right) = \sum_{i \leq 0, k>0} a_{ik}b_{ki} - \sum_{i>0,k\leq 0} a_{ik}b_{ki} \text{.} 
\end{equation}
As a special case of \eqref{eqn2.5}, we have that
\begin{equation}\label{eqn2.6} \mathrm{c}(\mathrm{E}_{ij},\mathrm{E}_{k\ell}) = \begin{cases}
-1 & \text{ $i=\ell > 0$, $j=k \leq 0$} \\
1 & \text{ $i=\ell \leq 0$, $j=k>0$} \\
0 & \text{otherwise,}
\end{cases} \end{equation}
for $i,j,k,\ell \in \ZZ$;
compare with \cite[p.~ 12]{Bloch:Okounkov:2000} or \cite[p.~ 115 and p.~ 313]{Kac:infinite:Lie:Algebras}.  Explicitly, as a $\CC$-vector space 
$$\widehat{{\bf gl}}(\infty) = \CC  \oplus {\bf gl}(\infty), $$ and the Lie bracket is defined by
$$[(a,x),(b,y)]=(\mathrm{c}(x,y),[x,y])\text{,}$$ for all $(a,x),(b,y) \in \widehat{{\bf gl}}(\infty)$.

\np{}\label{2.11}  We regard $R := \CC[t,t^{-1}]$, the ring of Laurent polynomials, as the loop algebra of the abelian Lie algebra $\CC$.  The \emph{oscillator algebra} is the Lie algebra $\mathfrak{s}$ determined by the $2$-cocycle with values in the trivial $R$-module $\CC$ given by:
$$\omega : \CC[t,t^{-1}] \times \CC[t,t^{-1}] \rightarrow \CC\text{,}$$
$$\omega\left((f(t),g(t) \right) := \operatorname{res}\left(\frac{df}{dt} g \right). $$  Concretely, 
$$\mathfrak{s} = \CC \oplus \CC[t,t^{-1}] \text{,}$$
and the bracket is given by
$$[(a,t^m),(b,t^n)]=(m\delta_{m,-n},0),$$ for all $a,b \in \CC$ and $m,n\in \ZZ$. 

\np{}\label{2.12} As in \cite[p.~ 313]{Kac:infinite:Lie:Algebras}, we realize the oscillator algebra $\mathfrak{s}$  as a subalgebra of $\widehat{\mathbf{gl}}(\infty)$ by the monomorphism of Lie algebras
\begin{equation}\label{oscillator algebra:mono}
\delta_0 :\mathfrak{s} \rightarrow \widehat{\mathbf{gl}}(\infty),
\end{equation} defined by  
$$(a,t^m) \mapsto \left(a,\sum_{j \in \ZZ} \E_{j,j+m}\right).$$

\section{The Lie algebra $\widehat{\mathbf{gl}}(\infty)$ and universal extensions}\label{3}

In this section we establish Theorem \ref{universal:extension:theorem} which shows how  the Lie algebra $\widehat{\mathbf{gl}}(\infty)$ is related to the Lie algebra $\widehat{g}$ which we define to be the \emph{universal (central) extension} of $g[t,t^{-1}]$ the Loop algebra of $g$ a complex finite dimensional semi-simple Lie algebra.  

\np{}\label{3.1}
Let $g$ be a complex finite dimensional semi-simple Lie algebra and $\kappa(\cdot,\cdot)$ its killing form.  We denote by $\widehat{g}$ the \emph{universal extension} of $g[t,t^{-1}]$.  Then $\widehat{g}$ is the central extension determined by the $2$-cocycle 
\begin{equation}\label{eqn3.1} u(\cdot,\cdot) : g[t,t^{-1}]\times g[t,t^{-1}] \rightarrow \CC
\end{equation} 
defined by
\begin{equation}\label{eqn3.2} 
u\left(\sum t^i\otimes x_i, \sum t^j \otimes y_j\right) :=\sum i \kappa(x_i,y_{-i})\text{,} 
\end{equation} 
\cite[\S 2]{Garland:1980} see also \cite[\S 7.9]{Weibel} especially \cite[\S 7.9.6, p.~ 250]{Weibel}.

To relate $\widehat{g}$ and $\widehat{\mathbf{gl}}(\infty)$, we choose a basis for $g$ and then consider its \emph{extended adjoint representation}:
\begin{equation}\label{eqn3.3} 1\otimes \operatorname{ad} : g[t,t^{-1}] \rightarrow \mathbf{gl}(\infty)\text{,}
\end{equation}
see  \eqref{eqn3.4} below.

The morphism $1\otimes \operatorname{ad}$, given by \eqref{eqn3.3}, allows us to compare the pullback of $\widehat{\mathbf{gl}}(\infty)$, with resect to $1\otimes \operatorname{ad}$, and the universal extension $\widehat{g}$.  In \S \ref{Sec3.3}, we prove:

\begin{theorem}\label{universal:extension:theorem}
The universal central extension of $g[t,t^{-1}]$ is the pull-back of $\widehat{\mathbf{gl}}(\infty)$ via $1\otimes \operatorname{ad}$, the extended adjoint representation of $g$.
\end{theorem}

\np{}\label{Sec3.2}
Before proving Theorem \ref{universal:extension:theorem} we first observe:

\begin{proposition}\label{cocycle:calc} The pullback of $\operatorname{c}(\cdot,\cdot)$ to  $\mathfrak{gl}_N[t,t^{-1}]$ via $\iota_N$ is given by: 
\begin{equation}\label{eqn3.4} \operatorname{c}(\iota_N(t^m\otimes x),\iota_N(t^n\otimes y)) = m \delta_{m,-n}\operatorname{tr}(xy)\text{.} 
\end{equation}
\end{proposition}

\begin{proof}
In light of the map \eqref{loop:incl}, it suffices to check that, for fixed $N \in \ZZ_{\geq 1}$, $1 \leq i,j,k,\ell \leq N$, $m,n\in\ZZ$, we have
\begin{multline*}
\mathrm{c} \left( \sum_{p \in \ZZ} \E_{N(p-m)+i, Np+j}, \sum_{q \in \ZZ} \E_{N(q-n)+k, Nq+\ell} \right) \\ =  
\begin{cases}
m & \text{if $j=k$, $i=\ell$ and $m=-n$.} \\
0 & \text{ otherwise.}
\end{cases}
\end{multline*}

 To compute 
$$\mathrm{c} \left( \sum_{p \in \ZZ} \E_{N(p-m)+i, Np+j}, \sum_{q \in \ZZ} \E_{N(q-n)+k, Nq+\ell} \right)\text{,}
$$ considering \eqref{eqn2.5}, it is clear that 
we need to understand the quantity:
\begin{multline}\label{important:quantity}
\sum_{\substack{p,q \in \ZZ, \\
Np+j > 0, \\
k+N(q-n)>0, \\
i+N(p-m)\leq 0,  \\
Nq+\ell \leq 0}} \delta_{Np+j,N(q-n)+k} \delta_{N(p-m)+i,Nq+\ell} 
\\
- \sum_{\substack{p,q \in \ZZ, \\ 
Np+j \leq 0, \\
k+N(q-n) \leq 0, \\
i+N(p-m) > 0,  \\
Nq+\ell > 0}} \delta_{Np+j,N(q-n)+k} \delta_{N(p-m)+i,Nq+\ell}\text{.}
\end{multline}

To this end, we make the following deductions:
\begin{enumerate}
\item{if \eqref{important:quantity} is nonzero, then $m=-n$;}
\item{if $m\geq 0$ the first sum appearing in \eqref{important:quantity} is nonzero if and only if $j=k$ and $i=\ell$, while the second sum is zero; the nonzero summands appearing \eqref{important:quantity}, when $j=k$ and $i=\ell$, are in bijection with the set of pairs $(p,q) \in \ZZ \times \ZZ$ with $-1\leq p \leq m$ and $n\leq q < 0$;}
\item{if $m<0$, the second sum appearing in \eqref{important:quantity} is nonzero if and only if $j=k$ and $i=\ell$, while the first sum is zero; the nonzero summands appearing in \eqref{important:quantity}, when $j=k$ and $i=\ell$, are in bijection with the set of pairs $(p,q) \in \ZZ \times \ZZ$ with $m\leq p < 0$, $0\leq q < n$.}
\end{enumerate}
The conclusion of Proposition \ref{cocycle:calc} follows immediately from these deductions. 
\end{proof}

\np{}\label{Sec3.3} We now establish Theorem \ref{universal:extension:theorem}.  To do so, first consider an arbitrary semi-simple Lie algebra $g$ and its adjoint representation
$$ \operatorname{ad} : g \rightarrow \operatorname{End}_\CC(g)\text{.}$$
Let $N=\dim_\CC g$ and fix a basis for $g$.  By composition we obtain a representation 
$$ \operatorname{ad} : g \rightarrow \operatorname{End}_\CC(g) \xrightarrow{\sim} \mathfrak{gl}_N(\CC),$$
which we can use to define the \emph{extended adjoint representation} of $g$
\begin{equation}\label{eqn3.4} g[t,t^{-1}] \xrightarrow{1\otimes \operatorname{ad}} \mathfrak{gl}_N(\CC)[t,t^{-1}] \xrightarrow{\iota_N} \mathbf{gl}(\infty)\text{.} 
\end{equation}

The homomorphism \eqref{eqn3.4}  allows us to compare the pull-back of $\widehat{\mathbf{gl}}(\infty)$, via $1 \otimes \operatorname{ad}$, with $\widehat{g}$.

\begin{proof}[Proof of {Theorem \ref{universal:extension:theorem}}]
It is enough to show that 
$$u\left(\sum t^i \otimes x_i, \sum t^j \otimes y_j  \right) :=\sum i \kappa(x_i,y_{-i}) $$
equals
$$\mathrm{c}\left(\sum t^i\otimes \operatorname{ad}x_i, \sum t^j \otimes \operatorname{ad}y_j \right) \text{.} $$
That this equality holds true follows from the fact that 
$$\kappa(x,y):=\operatorname{tr}(\operatorname{ad}x\operatorname{ad}y)$$ 
and from Proposition \ref{cocycle:calc}.
\end{proof}

\section{Semi-infinite monomials and the infinite wedge representation}\label{4}

In this section we study certain subsequences of $\ZZ$ which we refer to as \emph{semi-infinite monomials}, see \S \ref{4.1}.  We then describe the \emph{infinite wedge space} and the \emph{infinite wedge representation} of the Lie algebra $\widehat{\mathbf{gl}}(\infty)$, see \S \ref{4.6} and \S \ref{4.8} respectively.  What we do here is influenced heavily by what is done in \cite{Kac:infinite:Lie:Algebras}, \cite{Segal:Wilson:85}, \cite{jimbo:miwa:1983} and \cite{Miwa:Jimbo:Date}.  We give proofs of all assertions for completeness and because they are needed in our proof of Theorem \ref{bosonic:Extension:theorem}.

\np{}\label{4.1}
By a \emph{semi-infinite monomial} we mean an ordered strictly decreasing sequence of integers $S=(s_1,s_2,\dots)$, $s_i \in \ZZ$, with the properties that $s_i = s_{i-1}-1$ for all $i \gg 0$.   We let $\mathscr{S}$ denote the set of semi-infinite monomials.  If $S \in \mathscr{S}$, then define strictly decreasing sequences of integers $S_+$ and $S_-$ by $S_+:= S \setminus \ZZ_{\leq 0}$ and $S_- := \ZZ_{\leq 0} \setminus S$.  

\np{}\label{4.2}If $S = (s_1,s_2,\dots) \in  \mathscr{S}$, then there exists a unique integer $m$ with the property that $s_i = m-i+1$ for all $i \gg 0$.  We refer to this number as the \emph{charge} of $S$ and denote it by $\operatorname{charge}(S)$, compare with \cite[p.~ 12]{Segal:Wilson:85}, \cite[p.~ 310]{Kac:infinite:Lie:Algebras}, and \cite[A.3]{Okounkov:2001:partitions} for instance.  If $m \in \ZZ$, then let $\mathscr{S}_m := \{ S \in \mathscr{S} : \operatorname{charge}(S) = m \}$.  
We record the following proposition for later use.

\begin{proposition}\label{prop5.1}
The following assertions hold true:
\begin{enumerate}
\item{If $S \in \mathscr{S}$, then $\operatorname{charge}(S) = \# S_+ - \# S_-$;}
\item{Let $m \in \ZZ$.  The map $\lambda: \mathscr{S}_m \rightarrow \mathscr{P}$ defined by 
$$S=(s_1,s_2,\dots) \mapsto \lambda(S) = (\lambda_1,\lambda_2,\dots)\text{,} $$
where 
\begin{equation}  \lambda_j := s_j - m + j - 1, 
\end{equation} is a bijection.}
\end{enumerate}
\end{proposition}
\begin{proof}
To prove (a), let $S := (s_1,s_2,\dots) \in \mathscr{S}$, and write:
\begin{equation}\label{prop5.1.0}
 S_+ := (s_1,\dots, s_\ell) \\ 
\text{ and }
 S_- := (n_1,\dots, n_r); 
\end{equation} 
 here 
$s_1 > s_2 > \dots > s_\ell$ and $0 \geq n_1 > n_2 > \dots > n_r$.  Considering the definitions of $S_+$ and $S_-$ we deduce that
\begin{equation}\label{prop5.1.1}
s_{n+k} = n_r - k \text{ for $n := \ell-n_r-r+1$ and $k \geq 1$.}
\end{equation}
Now suppose that $m := \ell - r = \# S_+ - \# S_-$ and let $i = n + k$ for $k \geq 1$.  We then have
$ m-i+1= n_r - k$
which equals $s_i$ by \eqref{prop5.1.1}.  Conversely, suppose that $s_i = m-i+1$ for all $i \gg 0$.  We then have for all $k \gg0$ that
\begin{equation}\label{prop5.1.2}
s_{n+k} = m-n-k+1 = m-\ell+n_r+r-k.
\end{equation}
Combining \eqref{prop5.1.1} and \eqref{prop5.1.2}, we then have
\begin{equation}
n_r-k = m-\ell +n_r+r-k
\end{equation}
and so $m=\ell-r$ as desired.

For (b), first note that the map $\lambda$ is clearly injective.  To see that it is surjective, if $\lambda = (\lambda_1,\lambda_2,\dots) \in \mathscr{P}$, then define an element $S = (s_1,s_2,\dots) \in \mathscr{S}_m$ by $s_j = \lambda_j + m - j + 1$.  By construction $S \in \mathscr{S}$.  To see that $S \in \mathscr{S}_m$ note that $s_j = m-j + 1$ for $j > \operatorname{length}(\lambda)$.  
\end{proof}

\np{}\label{4.3}{\bf Remark.}  When we express $S \in \mathscr{S}$ as in \eqref{prop5.1.0}, the length of the partition $\lambda(S)$ equals the number $n$ defined in \eqref{prop5.1.1}.
Also the weight of the partition $\lambda(S)$ is sometimes referred to as the \emph{energy} of  $S$, \cite[p.~ 310]{Kac:infinite:Lie:Algebras}.

\np{}\label{4.4}{\bf Example.}
We can use the approach of \cite[\S 7.2]{Mike:covers} to give a graphical interpretation of Proposition \ref{prop5.1} for the case $m=0$.  The case $m \not = 0$ can be handled similarly with a shift.  As an example, the Young diagram associated to the partition $\lambda = (4,4,3,3,2,2,1) \in \mathscr{P}_{19}$ is:
$$
 \begin{Young}
&&&\cr \\
&&&\cr \\
&&\cr \\
&&\cr \\
&\cr \\
&\cr \\
\cr
\end{Young}
$$
If we cut this Young diagram along the main diagonal then there are $3$ rows in the top piece and $3$ columns in the bottom piece.  Let $u_i$, $i=1,2,3$, denote the number of boxes in the $i$-th row of the top piece and let $v_i$, $i=1,2,3$, denote the number of boxes in the $i$-th column of the bottom piece.  Then, $u_1 = 3.5$, $u_2 = 2.5$, $u_3 = .5$ and $v_1 = 6.5$, $v_2 = 4.5$, and $v_3 = 1.5$.

If $S$ is the charge zero semi-infinite monomial corresponding to $\lambda$, then $S$ is determined by the condition that 
$$S_+=(u_1+.5, u_2 + .5, u_3 + .5) = (4,3,1)$$ and 
$$S_-=(-v_3+.5, -v_2 + .5, -v_1 + .5) = (-1,-4, -6)\text{.}$$
In other words, 
$$S=(4,3,1,0,-2,-3,-5,-7,-8,\dots)$$
is the element of $\mathscr{S}_0$ corresponding to the partition 
$$\lambda = (4,4,3,3,2,2,1).$$

We can also relate the set $S$ to the \emph{code}, in the sense of \cite[\S 2]{Carrell:Goulden:2010}, of the partition $\lambda$.  Specifically, if $n \in \ZZ$, $n\geq 1$ and $n\not \in S$, then $n$ corresponds to an $R$; if $n \geq 1$ and $n \in S$, then $n$ corresponds to a $U$.  If $n \in \ZZ$, $n\leq0$ and $n \in S$, then $n$ corresponds to a $U$; if $n \leq 0$ and $n \not \in S$, then $n$ corresponds to an $R$.  The string consisting of these $R$'s and $U$'s is the code corresponding to $\lambda$ and our set $S$.

\np{}\label{4.5}
Let $\lambda:\mathscr{S} \rightarrow \mathscr{P}$ denote the extension of the bijections $\lambda : \mathscr{S}_m \rightarrow \mathscr{P}$ described in Proposition \ref{prop5.1} (b).  Also, to keep track of various minus signs which appear in what follows, we make the following definition: if $S \in \mathscr{S}$ and $j \in \ZZ$, then define $\operatorname{count}(j,S)$ to be the number of elements of $S$ that are strictly greater than $j$, that is:
\begin{equation}\label{5.5}
\operatorname{count}(j,S) := \# \{s \in S : j < s \} \text{.} 
\end{equation}

\np{}\label{4.6}
The \emph{infinite wedge space} is the $\CC$-vector space 
$F:=\bigoplus_{S \in \mathscr{S}} \CC$ determined by the set $\mathscr{S}$, see for instance \cite[\S 14.15]{Kac:infinite:Lie:Algebras} or \cite[p.~ 76]{Okounkov:2001:partitions}. In particular,
$$F = \operatorname{span}_\CC \{v_S : S \in \mathscr{S} \} $$ where $v_S =(r_T)_{T \in \mathscr{S}}$ denotes the element of $F$ given by $r_T=0$ for $T \not = S$ and $r_S=1$.  
If $m \in \ZZ$, then let 
$F^{(m)} := \Span_\CC \{ v_S : S \in \mathscr{S}_m \}\text{.} $ We then have $F= \bigoplus_{m \in \ZZ} F^{(m)}$, compare with \cite[p.~ 310]{Kac:infinite:Lie:Algebras}.

\np{}\label{fermi:creation:ann:1}%
We now recall the definition of \emph{wedging} and \emph{contracting} operators.  Our approach here is only notationally different from that of \cite[p.~ 311]{Kac:infinite:Lie:Algebras}.  On the other hand, we find our approach useful for relating these operators to our combinatorial construction on partitions, see \S \ref{fermi:gl} and especially Proposition \ref{prop6.2}.

To begin with, if $S = (s_1,s_2,\dots)$ is an ordered strictly decreasing sequence of integers and $j \in \ZZ$, then we use the notations $S\cup \{j\}$ and $S \setminus \{j\}$ to denote the ordered strictly decreasing sequence of integers determined by the sets $\{s_1,s_2,\dots\} \cup\{j\}$ and $\{s_1,s_2,\dots\} \setminus \{j\}$ respectively.
  
Next, given $j \in \ZZ$, define elements $f_j, f_j^* \in \operatorname{End}_\CC(F)$, for $j \in \ZZ$, by:
\begin{equation}\label{creation1}
f_j(v_S) := \begin{cases}
(-1)^{\operatorname{count}(j,S)} v_{S \cup \{j\}} & \text{if $j\not \in S$} \\
0 & \text{ if $j \in S$} 
\end{cases}
\end{equation}
and
\begin{equation}\label{ann1}
f^*_j(v_S) :=
\begin{cases}
(-1)^{\operatorname{count}(j,S)}v_{S \setminus \{j\}} & \text{ if $j \in S$} \\
0 & \text{ if $j \not \in S$,}
\end{cases}
\end{equation}
and extending $\CC$-linearly, compare with \cite[\S 14.17]{Kac:infinite:Lie:Algebras}, \cite[p.~ 12]{Bloch:Okounkov:2000}, and \cite[\S A]{Bloch:Okounkov:2000}.
These endomorphisms have the properties that
\begin{equation}\label{eq4.9}
\text{ $f_if_j^*+f_j^*f_i = \delta_{ij}$, $f_if_j+f_jf_i=0$,  $f_i^*f_j^*+f_j^*f_i^* =0$}
\end{equation}
and
\begin{equation}\label{eq4.10}
[f_if_j^*, f_\ell f^*_k] = \delta_{j\ell}f_if_k^*-\delta_{ik}f_\ell f_j^*,
\end{equation}
for all $i,j,k,\ell \in \ZZ$,
see \cite[p.~ 311]{Kac:infinite:Lie:Algebras} for example. 

For completeness, we note that \eqref{eq4.9} follows immediately from the definitions given in \eqref{creation1} and \eqref{ann1}.  On the other hand,  \eqref{eq4.10} is a consequence of \eqref{eq4.9}.  Indeed, first note:
$$ [f_if_j^*, f_\ell f_k^*]=f_if_j^* f_\ell f_k^* - f_i f_\ell f_k^* f_j^*+f_if_\ell f_k^* f_j^* - f_\ell f_k^* f_if_j^*$$
which can be rewritten using the second and third properties of \eqref{eq4.9} as:
\begin{equation}\label{fermi:calc:1}
f_i(f_j^*f_\ell + f_\ell f_j^*)f_k^*-f_\ell(f_if_k^*+f_k^*f_i)f_j^*\text{.}
\end{equation}
Applying the first property given in \eqref{eq4.9} to \eqref{fermi:calc:1} yields the righthand side of \eqref{eq4.10}.

 Note also that the operators $f_i$, for $i\in\ZZ$, map $F^{(m)}$ to $F^{(m+1)}$, the operators $f_i^*$, for $i\in\ZZ$, map $F^{(m)}$ to $F^{(m-1)}$ whereas the operators $f_i f_j^*$, for $i,j\in\ZZ$, map $F^{(m)}$ to $F^{(m)}$.

\np{}\label{4.8}
The \emph{infinite wedge representation} is the Lie algebra homomorphism 
$$\rho :\widehat{\mathbf{gl}}(\infty) \rightarrow \End_\CC(F)$$ determined by the conditions that
\begin{equation}\label{4.8'} \rho((0,\E_{ij})) = \begin{cases} f_if_j^* & \text{ if $i\not = j$ or $i=j>0$} \\
f_if_i^* - \operatorname{id}_F & \text{ if $j=i \leq 0$}
\end{cases} 
\end{equation}
and 
\begin{equation}\label{4.8''}
\rho((a,0))=a \operatorname{id}_F, 
\end{equation}
for $a \in \CC$, compare with \cite[p.~ 313]{Kac:infinite:Lie:Algebras} for instance.

The fact that the above conditions \eqref{4.8'} and \eqref{4.8''} determine a representation of Lie algebras is deduced easily from property \eqref{eq4.10} above together with the definition of the $2$-cocycle $\operatorname{c}(\cdot,\cdot)$, given in \eqref{eqn2.6}, and the fact that every element of $\mathbf{gl}(\infty)$ can be written as a finite linear combination of matrices of the form 
$\sum_{i\in\ZZ}a_i \E_{i,i+k}\text{,}$ where $k \in \ZZ$ and $a_i \in \CC$.

\np{}\label{4.9}
In what follows we refer to the restriction of $\rho$ to the image of the morphism  \eqref{oscillator algebra:mono} as the \emph{infinite wedge representation of the oscillator algebra} $\mathfrak{s}$.

\section{Combinatorial properties of the operators $f_i f_j^*$}\label{fermi:gl}

In this section we define and study certain operators on partitions.  This construction will be used in our definition of the bosonic representation of the Lie algebra $\widehat{\mathbf{gl}}(\infty)$, see \S \ref{6}.   Our main result is Proposition \ref{prop6.2} which describes the combinatorics encoded in the vector
\begin{equation}\label{psi:p:q:S:eqn1} f_if_j^*(v_S) = (-1)^{\alpha} v_{T} \text{;}
\end{equation}
here 
$$S := (s_1,s_2,\dots) \in \mathscr{S},$$  $i,j \in \ZZ$, are such that 
$$\text{$j \in S$ and $i \not \in S \setminus \{j\}$,}$$
$$T := (S \setminus \{j \}) \cup \{i\},$$ and 
$$\alpha := \operatorname{count}(i, S \setminus \{j\}) - \operatorname{count}(j,S).$$

As it turns out the combinatorics encoded in  \eqref{psi:p:q:S:eqn1} are related to a certain skew diagram associated to the partition determined by $S$, see Proposition \ref{prop6.1} and Proposition \ref{prop6.2}.

\np{}\label{5.1}
Let $m,i\in\ZZ$, let $\mathscr{P}_{m,i}$ denote the set
$$\mathscr{P}_{m,i} := \{\lambda=(\lambda_1,\lambda_2,\dots) \in \mathscr{P} : \lambda_k \not = i -m+k-1 \text{ for all $k$} \} \text{,}$$
and let $\mathscr{P}_{m,i}^*$ denote the set
$$ \mathscr{P}_{m,i}^* := \{ \lambda = (\lambda_1,\lambda_2,\dots) \in \mathscr{P} : \lambda_k = i -m+k-1 \text{ for some $k$}\}.$$
Given $\lambda=(\lambda_1,\lambda_2,\dots) \in \mathscr{P}$, define 
\begin{equation}\label{6.1} \operatorname{count}_m(i,\lambda) := \# \{k : \lambda_k > i-m+k-1 \}\text{.}
\end{equation}
The main idea behind \eqref{6.1} is that if $\lambda = \lambda(S)$ is the partition corresponding to a charge $m$ semi-infinite monomial $S \in \mathscr{S}_m$, then 
\begin{equation}\label{6.2}
\operatorname{count}_m(i,\lambda) = \operatorname{count}(i,S),
\end{equation} 
where $\operatorname{count}(i,S)$ denotes the number of elements of $S$ which are strictly greater than $i$, see \eqref{5.5}.
That \eqref{6.2} holds true is easy to check using \eqref{5.5} and Proposition \ref{prop5.1} (b).

\np{}\label{5.2}  We now use \eqref{6.1} to define certain combinatorial operators on partitions.   Precisely, if $\lambda=(\lambda_1,\lambda_2,\dots) \in \mathscr{P}_{m,i}$, then define $p_{m,i}(\lambda)$ to be the partition $\mu=(\mu_1,\mu_2,\dots)$ where:
\begin{equation}\label{6.2'} \mu_j = \begin{cases} 
\lambda_j - 1 & \text{ for $j \leq \operatorname{count}_m(i,\lambda)$} \\
i-m+\operatorname{count}_m(i,\lambda) - 1 & \text{ for $j=\operatorname{count}_m(i,\lambda)+1$} \\
\lambda_{j-1} & \text{ for $j >\operatorname{count}_m(i,\lambda)+1$.}
\end{cases}
\end{equation}

On the other hand, if $\lambda=(\lambda_1,\lambda_2,\dots) \in \mathscr{P}_{m,i}^*$, then define $p_{m,i}^*(\lambda)$ to be the partition $\mu=(\mu_1,\mu_2,\dots)$ where: 
\begin{equation}\label{6.2''}\mu_j = \begin{cases} \lambda_j + 1 & \text{for $j \leq \operatorname{count}_m(i,\lambda)$} \\
\lambda_{j+1} & \text{for $j > \operatorname{count}_m(i,\lambda)$.
 }\end{cases} 
 \end{equation}
  
\np{}  The following proposition is used in the proof of Proposition \ref{prop6.2} which relates the combinatorial operators defined in \S \ref{5.2} to the operators $f_i f_j^*$ described in \S \ref{fermi:creation:ann:1} and \eqref{psi:p:q:S:eqn1}.

\begin{proposition}\label{prop6.1}
Fix $m,i,j\in \ZZ$, $\lambda = (\lambda_1,\lambda_2,\dots) \in \mathscr{P}_{m,j}^*$, let 
$$\mu := p_{m,j}^*(\lambda),$$ 
assume that $\mu \in \mathscr{P}_{m-1,i}$ and let 
$$\nu := p_{m-1,i}(\mu) = p_{m-1,i}p_{m,j}^*(\lambda).$$  
The following assertions hold true:
\begin{enumerate}
\item{if $i<j$, then $\nu \subseteq \lambda$, the skew diagram $\lambda \setminus \nu$ is a border strip, 
$$\# (\lambda \setminus \nu) = j - i,$$ and 
$$\operatorname{height}(\lambda \setminus \nu) = \operatorname{count}_{m-1}(i,\mu) - \operatorname{count}_m(j,\lambda);$$}
\item{if $i > j$, then $\lambda \subseteq \nu$, the skew diagram $\nu \setminus \lambda$ is a border strip, 
$$\#(\nu \setminus \lambda) = i-j,$$ and 
$$\operatorname{height}(\nu \setminus \lambda) = \operatorname{count}_m(j,\lambda) - \operatorname{count}_{m-1}(i,\mu).$$}
\item{if $i = j$, then $\nu = \lambda$ and the skew diagrams $\nu \setminus \lambda $ and $\lambda \setminus \nu$ are empty.}
\end{enumerate}
\end{proposition}
\begin{proof}
By assumption we have
\begin{equation} \mu := p_{m,j}^*(\lambda)
\end{equation}
and
\begin{equation} 
\nu := p_{m-1,i}(\mu) = p_{m-1,i}p_{m,j}^*(\lambda)= (\nu_1,\nu_2,\dots);
\end{equation}
set
\begin{equation}\label{6.4}
\alpha := \operatorname{count}_m(j,\lambda) \end{equation}
and  
\begin{equation}\label{6.5}
\beta := \operatorname{count}_{m-1}(i,\mu).
\end{equation}

For (a), we have $i < j$.  As a consequence, using the definitions \eqref{6.2'} and \eqref{6.2''},
we deduce that the partition $\nu = (\nu_1,\nu_2,\dots)$ has the form:
\begin{equation}\label{6.5}
\nu_k =  \begin{cases}
\lambda_k & \text{ for $1 \leq k \leq \alpha$ } \\
\lambda_{k+1} - 1 & \text{ for $\alpha + 1 \leq k \leq \beta$} \\
i-m+\beta & \text{ for $k = \beta + 1$} \\
\lambda_k & \text{ for $k \geq \beta + 2$.}
\end{cases}
\end{equation}

Considering \eqref{6.5}, it is clear that $\nu \subseteq \lambda$, that $\theta := \lambda \setminus \nu$ is a border strip, and that the number of rows of $\theta$ equals 
\begin{equation}\label{6.5'}\#[\alpha+1,\beta+1] = \beta - \alpha + 1;
\end{equation} 
it follows from \eqref{6.5'} that 
\begin{equation}\label{6.5''}
\operatorname{height}(\theta) = \beta - \alpha.
\end{equation}

Next if $\theta_k$ denotes the number of elements in the $k$th row of $\theta$, then $\theta_k = 0$ for $k \leq \alpha$ and $k \geq \beta+2$.  We also have:
\begin{equation}\label{6.6}
\theta_k = \lambda_k - \lambda_{k+1} + 1, 
\end{equation}
 for $\alpha + 1 \leq k \leq \beta$,
\begin{equation}\label{6.7}
\theta_{\beta + 1} = \lambda_{\beta+1} - i -\beta +m\text{,}
\end{equation}
and 
\begin{equation}\label{6.8}
\lambda_{\alpha+1} = j + \alpha - m.
\end{equation}
Thus, using \eqref{6.6}, \eqref{6.7}, and \eqref{6.8}, we have:
$$\sum_{k=\alpha+1}^{\beta+1} \theta_k = j + \alpha -m - i - \beta + m + \#[\alpha+1,\beta] = j - i,
$$
whence
$$\# \theta = j-i.
$$

For (b), we have $i > j$.  As a consequence, using the definitions \eqref{6.2'} and \eqref{6.2''}, we deduce that the partition $\nu = (\nu_1,\nu_2,\dots)$ is defined by:
\begin{equation}\label{6.11}
\begin{cases}
\lambda_k & \text{ for $1 \leq k \leq \beta$} \\
i+ \beta - m & \text{ for $k = \beta + 1 $} \\
\lambda_{k-1}+1 & \text{ for $\beta + 1 < k \leq \alpha+1$} \\
\lambda_k & \text{ for $k \geq \alpha+2$.}
\end{cases}
\end{equation}
Considering \eqref{6.11}, it is clear that $\lambda \subseteq \nu$, that $\theta := \nu \setminus \lambda$ is a border strip, and that the number of rows of $\theta$ equals 
\begin{equation}\label{6.12}
\#[\beta+1,\alpha+1]=\alpha-\beta+1.
\end{equation}  
Thus 
\begin{equation}\label{6.13}
\operatorname{height}(\theta) = \alpha - \beta.
\end{equation}

Next let $\theta_k$ denote the number of elements in the $k$th row of $\theta$.  Then $\theta_k = 0$ for $k \leq \beta$ and $k > \alpha + 1$.  We also have:
\begin{equation}\label{6.14}
\theta_{\beta+1} = i + \beta -m - \lambda_{\beta+1},
\end{equation}
\begin{equation}\label{6.14'}
\theta_k = \lambda_{k-1}+1-\lambda_k \text{,}
\end{equation}
 for $\beta+1 < k \leq \alpha+1$,
and
\begin{equation}\label{6.14''}
\lambda_{\alpha+1} = j + \alpha - m.
\end{equation}
Using \eqref{6.14}, \eqref{6.14'}, and \eqref{6.14''}, it follows that
$$\sum_{k = \beta + 1}^{\alpha+1} \theta_k = i + \beta - m - j - \alpha + m + \#[\beta+2,\alpha+1] = i - j
$$
so that 
$$\# \theta = i - j.
$$
Assertion (c) is trivial.
\end{proof}

\subsection{}\label{ex6.2}{\bf Example.}  
Recall, see \S \ref{4.4}, that 
$$S = (4,3,1,0,-2,-3,-5,-7,-8,\dots)$$ is the element of $\mathscr{S}_0$ corresponding to the partition 
$$\lambda = (4,4,3,3,2,2,1) \in \mathscr{P}_{19},$$ whose Young diagram is pictured in \S \ref{4.4}. 
To compute $f_{-1}f_3^*(v_S)$ note that 
$$T = (S\setminus \{3\}) \cup \{-1\} = (4,1,0,-1,-2,-3,-5,-7,-8,\dots)\text{,}$$ $\operatorname{count}(3,S)=1$ and $\operatorname{count}(-1,S\setminus \{3\})=3$.  We conclude
\begin{equation}\label{skew:diagram:eg1} f_{-1}f_3^*(v_S) = (-1)^{3-1}v_{T} = v_{T}\text{.}
\end{equation}

To see the combinatorics encoded in \eqref{skew:diagram:eg1}  first note that if 
$$\nu := \lambda(T),$$ the partition corresponding to $T$, then 
$$\nu = (4,2,2,2,2,2,1)$$ which has Young diagram
$$ \begin{Young}
&&&\cr \\
&\cr \\
&\cr \\
&\cr \\
&\cr \\
&\cr \\
\cr
\end{Young}$$
and $\nu \subseteq \lambda$.  The skew diagram $\theta := \lambda \setminus \nu$ is the set 
$$\{ \{2,3\}, \{2,4\}, \{3,3\}, \{4,3\} \} $$
which can be represented pictorially as:
$$\begin{Young}
&\cr \\
\cr \\
\cr 
\end{Young} $$ 
Note that the skew diagram $\theta$ is a border strip and $\operatorname{height}(\theta) =2$. If we now identify $S$ with $\lambda$ and $T$ with $\nu$, then \eqref{skew:diagram:eg1} takes the form
\begin{equation}\label{skew:diagram:eg1'} f_{-1}f_{3}^*(v_{\lambda}) = (-1)^{\operatorname{height}(\theta)} v_{\nu} \text{.}
\end{equation}

Suppose now that we wish to compute $f_{-1}f_{-3}^*(v_{S})$.  In this case,  $\operatorname{count}(-3,S)=5$, $\operatorname{count}(-1,S\setminus \{-3\})=4$ and hence 
\begin{equation}\label{skew:diagram:eg2} 
f_{-1}f_{-3}^*(v_S)=-1v_{T}\text{,}
\end{equation}
where 
$$T = (4,3,1,0,-1,-2,-5,-7,-8,\dots).$$

The combinatorics encoded in \eqref{skew:diagram:eg1'} is similar to that encoded in \eqref{skew:diagram:eg1}, but there is one difference which amounts to the fact that $-1 > -3$ while $3 > -1$.  In more detail, if 
$$\nu := \lambda(T),$$ then 
$$\nu = (4,4,3,3,3,3,1),$$ $\lambda \subseteq \nu$, and the skew diagram $\theta := \nu \setminus \lambda$ is 
$$\theta = \{ \{5,3\}, \{6,3\} \}$$ which is a border strip.  The border strip $\theta$ can be pictured pictorially as:
$$ 
\begin{Young}
\cr \\
\cr
\end{Young}
$$
and has height equal to one.
If we identify $S$ with $\lambda$ and $T$ with $\nu$, then \eqref{skew:diagram:eg2} takes the form
$$f_{-1}f_{-3}^*(v_\lambda) = (-1)^{\operatorname{height}(\theta)} v_{\nu}\text{.} $$

\np{} Example \ref{ex6.2} generalizes: 

\begin{proposition}\label{prop6.2}
Suppose that $S = (s_1,s_2,\dots) \in \mathscr{S}$ and $i,j \in \ZZ$.  Then $f_if_j^*(v_S) \not = 0$ if and only if $j \in S$, and $i \not \in S \setminus \{j\}$.  In addition assume that $f_i f_j^*(v_S) \not = 0$, let $T := (S \setminus \{j\})\cup \{i\}$, let $\lambda$ and $\nu$ be the partitions determined by $S$ and $T$ respectively, and denote $v_S$ by $v_\lambda$ and $v_{T}$ by $v_\nu$.  The following assertions hold true:
\begin{enumerate}
\item{If $i < j$, then $\nu \subseteq \lambda$, the skew diagram $\lambda \setminus \nu$ is a border strip of length $j-i$ and
$$f_i f_j^*(v_\lambda) = (-1)^{\operatorname{height}(\lambda \setminus \nu)} v_{\nu} \text{;} $$ }
\item{If $j<i$, then $\lambda \subseteq \nu$, the skew diagram $\nu \setminus \lambda$ is a border strip of length $i-j$ and
$$f_if_j^*(v_\lambda) = (-1)^{\operatorname{height}(\nu \setminus \lambda)}v_\nu \text{.} $$
}
\end{enumerate}
\end{proposition}
\begin{proof}
The proposition is a consequence of Proposition \ref{prop5.1}, the discussion given in \S \ref{psi:p:q:S:eqn1} and  Proposition \ref{prop6.1}.  In particular, using Proposition \ref{prop5.1} (b) in conjunction with \eqref{6.5} and \eqref{6.11}, depending on whether $i<j$ or $j<i$, we compute that 
$$\nu = p_{i,m-1}p_{j,m}^*(\lambda).$$  The conclusion of Proposition \ref{prop6.2} then follows from Proposition \ref{prop6.1}, \eqref{6.2} and  \eqref{psi:p:q:S:eqn1}.
\end{proof}

\section{The bosonic representation of $\widehat{\mathbf{gl}}(\infty)$}\label{6}

We now provide an application of our combinatorial construction given in \S \ref{fermi:gl}.  Indeed, we use this construction to prove the boson-fermion correspondence which we state as Theorem \ref{bosonic:Extension:theorem}.

\np{}\label{Sec6.1}  To begin with, let $A:=\CC[z,z^{-1}]$ and 
$$B := A\otimes_\CC \Lambda =\CC[z,z^{-1},h_1,h_2,\dots].$$  The \emph{bosonic representation of the oscillator algebra} is the Lie algebra homomorphism
\begin{equation}\label{eqn6.1} \xi_0: \mathfrak{s} \rightarrow \End_\CC(B)
\end{equation}
determined by:
$$ \xi_0((0,t^k))=p_k^\perp = k \frac{\partial}{\partial p_k} \text{, for $k > 0$;}$$
$$\xi_0((0,t^k))=p_{-k} \text{, for $k<0$;} $$
$$ \xi_0((0,1))=z\frac{\partial}{\partial z} \text{;}$$
and
$$ \xi_0((1,0))=1\text{,}$$
compare with \cite[p.~ 314]{Kac:infinite:Lie:Algebras} or \cite[Lecture 5, p.~ 46]{kac-raina-rozhkovskaya}.

\np{}\label{Sec6.2}  The first step to proving Theorem \ref{bosonic:Extension:theorem} is to define operators 
$$b_i \in \operatorname{End}_{\CC}(B)$$ 
by the rule:
\begin{equation}\label{eqn6.2}
b_i(z^m s_\lambda) = 
\begin{cases}
(-1)^{\operatorname{count}_m(i,\lambda)} z^{m+1}s_{p_{m,i}(\lambda)} & \text{ for $\lambda \in \mathscr{P}_{m,i}$} \\
0 & \text{ for $\lambda \not \in \mathscr{P}_{m,i}$.}
\end{cases}
\end{equation}
Similarly define operators $$b_i^* \in \operatorname{End}_{\CC}(B)$$ by the rule
\begin{equation}\label{eqn6.3'}
b_i^*(z^m s_\lambda)
\begin{cases}
(-1)^{\operatorname{count}_m(i,\lambda)} z^{m-1}s_{p_{m,i}^*(\lambda)} & \text{ for $\lambda \in \mathscr{P}^*_{m,i}$} \\
0 & \text{ for $\lambda \in \mathscr{P}^*_{m,i}$.}
\end{cases}
\end{equation}

As in \eqref{eq4.9} and \eqref{eq4.10}, we have the relations
\begin{equation}\label{eqn6.3} \text{ $b_i b_j^* + b_j^* b_i = \delta_{ij}$, $b_i b_j + b_j b_i = 0$, $b^*_i b_j^* + b_j^* b_i^* = 0$,}
\end{equation}
 and
\begin{equation}\label{eqn6.4}
[b_i b^*_j, b_{\ell} b^*_k] = \delta_{j \ell} b_i b_k^* - \delta_{i k} b_\ell b^*_j,
\end{equation}
for all $i,j,k,\ell \in \ZZ$.  Indeed, as in \eqref{eq4.9}, \eqref{eqn6.3} follows immediately from the definitions while \eqref{eqn6.4} is deduced from \eqref{eqn6.3}.

\np{}\label{Sec6.3} {\bf Example.}  As in \S \ref{4.4}, if $\lambda=(4,4,3,3,2,2,1)$, then 
$$\operatorname{count}_0(3,\lambda)=1$$ 
and 
$$b^*_3(s_\lambda) = -z^{-1}s_\mu\text{,}$$ where $\mu$ is the partition $\mu = (5,3,3,2,2,1)$.  Also, $\operatorname{count}_{-1}(-1,\mu) = 3$, 
$$b_{-1}(z^{-1}s_{\mu})=-s_{\nu}\text{,}$$ where $\nu = (4,2,2,2,2,1)$, and 
$$b_{-1}b_3^*(s_\lambda) = s_\nu\text{.}$$

\np{}\label{Sec6.4}
The key point in the proof of Theorem \ref{bosonic:Extension:theorem} is the following observation which is a consequence of Proposition \ref{prop6.2}.
The point is that if $S \in \mathscr{S}$, $k$ a nonzero integer and $\mathfrak{s}_k := (0,t^k) \in \mathfrak{s}$, then 
\begin{equation}\label{7.1} \delta_0(\mathfrak{s}_k)(v_S) = \sum_{\mathrm{finite}} f_\ell f^*_{\ell + k}(v_S)
\end{equation}
and we now give a combinatorial description of this finite set:

\begin{proposition}\label{prop7.2}
Suppose that $S \in \mathscr{S}_m$.  The following assertions hold true:
\begin{enumerate}
\item{ If $k>0$, then
$$ \mathfrak{s}_k(v_S) = \sum_{\mathrm{finite}}  (-1)^{\operatorname{height}(\lambda(S) \setminus \lambda(T))}v_{T} \text{,}$$
where the finite sum is taken over all
$T \in \mathscr{S}_m$, which have the property that $\lambda(T)\subseteq \lambda(S)$, and $\lambda(S) \setminus \lambda(T)$ is a border strip of length $k$. }
\item{If $k<0$, then
$$\mathfrak{s}_k(v_S) = \sum_{\mathrm{finite}} (-1)^{\operatorname{height}(\lambda(T) \setminus \lambda(S))}v_{T} $$
where the finite sum is taken over all $T \in \mathscr{S}_m$ with the property that $\lambda(S) \subseteq \lambda(T)$ and $\lambda(T) \setminus \lambda(S)$ is a border strip of length $|k|$.}
\end{enumerate}
\end{proposition}
\begin{proof}
To begin with, note that for both (a) and (b), Proposition \ref{prop6.2} implies that each summand of \eqref{7.1} contributes a summand of the desired form.

To establish Proposition \ref{prop7.2} it thus remains to show that, conversely, each border strip of the shape asserted in the proposition appears as a summand of \eqref{7.1}.

To this end, consider the case that $k>0$.  Let $\lambda$ be the partition corresponding to $S$, suppose that $\nu \subseteq \lambda$ is such that $\theta:= \lambda \setminus \nu$ is a border strip of length $k$.  Let $\theta_{n'}$ denote the number of elements in the $n'$th row of $\theta$.  Let $n:= \min \{ j : \nu_j \not = \lambda_j \}$.  Then $\theta_{n'} = 0$ for $n'<n$ and $\theta_n \not = 0$;  set
\begin{equation}\label{prop7.2.1}
\ell := \theta_n - k - n + m + \lambda_{n+1}.
\end{equation}
We then compute, using the definitions \eqref{6.2'} and \eqref{6.2''} together with the fact that $\theta$ is a border strip, that 
\begin{equation}\label{prop7.2.2}
\nu = p_{m-1,\ell}p^*_{m,\ell+k}(\lambda);
\end{equation}
compare with \eqref{6.6} and \eqref{6.7}.

Thus if $T$ is the element of $\mathscr{S}_m$ corresponding to $\nu$, then
\begin{equation}\label{prop7.2.3}
(-1)^{\operatorname{height}(\lambda \setminus \nu)} v_T = f_\ell f^*_{\ell+k}(v_S),
\end{equation}
by Proposition \ref{prop6.2} (a).

Next suppose that $k<0$.  Again let $\lambda$ be the partition corresponding to $S$, suppose that $\nu \supseteq \lambda$ is such that $\theta := \nu \setminus \lambda$ is a border strip of length $|k|$, let $T$ be the element of $\mathscr{S}_m$ corresponding to $\nu$ and let $n := \min \{ j : \nu_j \not = \lambda_j \}$.  Let $\theta_{n'}$ denote the number of elements in the $n'$th row of $\theta$ and set
\begin{equation}\label{prop7.2.3}
\ell :=  \theta_n - (n-1)+\lambda_n + m.
\end{equation}
We then compute, using the definitions \eqref{6.2'} and \eqref{6.2''} together with the fact that $\theta$ is a border strip, that:
\begin{equation}\label{prop7.2.4}
\nu = p_{m-1,\ell} p_{m,\ell+k}^*(\lambda);
\end{equation} compare with 
\eqref{6.14} and \eqref{6.14'}.  

Thus if $T$ is the element of $\mathscr{S}_m$ corresponding to $\nu$, then
\begin{equation}\label{prop7.2.4}
(-1)^{\operatorname{height}(\nu \setminus \lambda)} v_T = f_\ell f_{\ell+k}^*(v_S),
\end{equation}
by Proposition \ref{prop6.2} (b).
\end{proof}

\subsection{}\label{Sec6.5} 
Using the theory we have developed thus far we can prove the boson-fermion correspondence.

\begin{theorem}\label{bosonic:Extension:theorem}
The bosonic representation 
$$ \xi_0  : \mathfrak{s} \rightarrow \End_\CC(B),$$
namely \eqref{eqn6.1},
of the oscillator algebra extends to a representation 
$$\xi:\widehat{\mathbf{gl}}(\infty) \rightarrow \End_\CC(B)$$
of the Lie algebra $\widehat{\mathbf{gl}}(\infty)$. 
More precisely, the Lie algebra $\widehat{\mathbf{gl}}(\infty)$ admits a representation
$ \xi:\widehat{\mathbf{gl}}(\infty) \rightarrow \End_\CC(B)$
with the property that the diagram
$$
\begin{tikzcd}
\mathfrak{s} \arrow{dr}{\xi_0} \arrow{d}[swap]{\delta_0} & \\
\widehat{\mathbf{gl}}(\infty) \arrow{r}{\xi}& \End_\CC(B)
\end{tikzcd}
$$
commutes.
In addition, the $\CC$-linear isomorphism
$$\sigma : F\rightarrow B$$
\text{ defined by } 
$$v_S \mapsto z^m s_{\lambda(S)}, $$
for $m = \operatorname{charge}(S)$ and $\lambda(S)$ the partition determined by the semi-infinite monomial $S$,
is an isomorphism of 
$\widehat{\mathbf{gl}}(\infty)$-modules.  
\end{theorem}

\begin{proof}
Consider the representation 
$$\xi :\widehat{\mathbf{gl}}(\infty) \rightarrow \End_\CC(B)$$ 
determined by the conditions that:
$$\xi((0,\E_{ij})) = \begin{cases} b_ib_j^* & \text{ if $i\not = j$ or $i=j>0$} \\
b_ib_i^* - \operatorname{id}_B & \text{ if $j=i \leq 0$}
\end{cases} $$
and 
$$\xi((a,0))=a \operatorname{id}_B,$$ for $a \in \CC$.  The fact that $\xi$ is a representation follows from the relations given in \eqref{eqn6.3}. The fact that $\xi$ extends the representation $\xi_0$ follows from Proposition \ref{prop7.2} and the Murnaghan-Nakayama rule \eqref{Murnaghan:Nakayama} and \eqref{Murnaghan:Nakayama:adjoint}.

For the second assertion, fix $i,j\in \ZZ$ and assume that $f_if_j^*(v_S) \not = 0$.  We then have that $$f_if_j^*(v_S)=(-1)^\alpha v_T,$$
where 
$$T := (S \setminus \{j\})\cup \{i\},$$ and 
$$\alpha := \operatorname{count}(i,S\setminus \{j\}) - \operatorname{count}(j,S);$$  
let $\lambda = \lambda(S)$ be the partition corresponding to $S$ and let $\nu = \lambda(T)$ be the partition corresponding to $T$.

In this setting, the operator $p_{m,j}^*$ is defined on the partition $\lambda$ and the operator $p_{m-1,i}$ is defined on the partition $p_{m,j}^*(\lambda)$.  In addition 
$$\nu = p_{m-1,i}p_{m,j}^*(\lambda).$$  
On the other hand we have 
that 
$$\sigma(v_S)=z^m s_{\lambda}.$$  
Considering the definitions of the operators $b_i$ and $b_j^*$, we then deduce that 
$$b_i b_j^*(z^m s_\lambda)=(-1)^\alpha z^m s_\nu$$ which is what we wanted to show.
\end{proof}
 
\providecommand{\bysame}{\leavevmode\hbox to3em{\hrulefill}\thinspace}
\providecommand{\MR}{\relax\ifhmode\unskip\space\fi MR }
\providecommand{\MRhref}[2]{%
  \href{http://www.ams.org/mathscinet-getitem?mr=#1}{#2}
}
\providecommand{\href}[2]{#2}

\end{document}